\documentclass[12pt,a4paper,psamsfonts]{amsart}
\usepackage{amssymb,amscd,amsxtra,calc}
\usepackage{cmmib57}

\theoremstyle{plain}
    \newtheorem{thm}{Theorem}[section]

    \newtheorem{example}[thm]{Example}
    \newtheorem{lemma}[thm]{Lemma}

    \newtheorem{theorem}[thm]{Theorem}

\theoremstyle{definition}

    \newtheorem{remark}[thm]{Remark}
\theoremstyle{remark}

    \newtheorem{setup}[thm]{}

\newcommand{\BCC}{\mathbb{C}}
\newcommand{\BFF}{\mathbb{F}}

\newcommand{\PP}{\mathbb{P}}
\newcommand{\BPP}{\mathbb{P}}
\newcommand{\BQQ}{\mathbb{Q}}
\newcommand{\BRR}{\mathbb{R}}
\newcommand{\BZZ}{\mathbb{Z}}

\newcommand{\SO}{\mathcal{O}}

\newcommand{\divv}{\operatorname{div}}
\newcommand{\Gal}{\operatorname{Gal}}

\newcommand{\NE}{\operatorname{NE}}

\newcommand{\Pic}{\operatorname{Pic}}

\newcommand{\Sing}{\operatorname{Sing}}

\newcommand{\isom}{\simeq}
\newcommand{\ratmap}
{{\,\cdot\negmedspace\cdot\negmedspace\cdot\negmedspace\to\,}}

\begin{document}

\title[Invariant hypersurfaces]{
Invariant hypersurfaces of endomorphisms of the projective $3$-space}

\author{De-Qi Zhang
}
\address
{
\textsc{Department of Mathematics, National University of Singapore} \endgraf
\textsc{10 Lower Kent Ridge Road, Singapore 119076
}}
\email{matzdq@nus.edu.sg}

\begin{abstract}
We consider surjective endomorphisms $f$ of degree $> 1$ on the projective
$n$-space $\PP^n$ with $n = 3$,
and $f^{-1}$-stable hypersurfaces $V$. We show
that $V$ is a hyperplane (i.e., $\deg(V) = 1$) but
with four possible exceptions; it is conjectured that $\deg(V) = 1$ for any $n \ge 2$;
cf. \cite{FS}, \cite{BCS}.
$$
\begin{small}
\text{\it Dedicated to Prof. Miyanishi on the occasion of his 70th birthday}
\end{small}
$$
\end{abstract}


\thanks{The author is supported by an ARF of NUS}

\maketitle

\section{Introduction}

We work over the field $\BCC$ of complex numbers.
In this paper, we study properties of $f^{-1}$-stable prime divisors of $X$
for endomorphisms $f : \PP^3 \to \PP^3$.
Below is our main result.

\begin{theorem}\label{ThA}
Let $f : \BPP^3 \to \BPP^3$ be an endomorphism of degree $> 1$
and $V$ an irreducible hypersurface such that $f^{-1}(V) = V$.
Then either $\deg(V) = 1$, i.e., $V$ is a hyperplane, or $V$ equals one of the four cubic
hypersurfaces  $V_i = \{S_i = 0\}$, where $S_i$'s are as follows, with suitable projective coordinates:
\begin{itemize}
\item[(1)]
$S_1 = X_3^3 + X_0X_1X_2$;
\item[(2)]
$S_2 = X_0^2X_3 + X_0X_1^2 + X_2^3$;
\item[(3)]
$S_3 = X_0^2X_2 + X_1^2 X_3$;
\item[(4)]
$S_4 = X_0X_1X_2 + X_0^2X_3 + X_1^3$.
\end{itemize}
\end{theorem}

We are unable to rule out the four cases in Theorem \ref{ThA}
and do not know whether there is any endomorphism $f_{V_i} : V_i \to V_i$
of $\deg(f_{V_i}) > 1$ for $i = 2, 3$ or $4$,
but see Examples \ref{ex1ThA} (for $V_1$)
and \ref{ex2ThA} below.

\begin{example}\label{ex2ThA}
{\rm
There are many endomorphisms $f_{V'} : V' \to V'$ of $\deg(f_{V'}) > 1$ for the normalization
$V'$ of $V = V_i$ ($i = 3, 4$), where $V' \isom \BFF_1$ in either case (cf.~Remark $\ref{rThA}$ below).
Conjecture \ref{set1.1} below asserts that $f_{V'}$ is {\it not} lifted from
any endomorphism $f : \BPP^3 \to \BPP^3$
restricted to the non-normal cubic surface $V$.
Indeed, consider the endomorphism
$f_{\BPP^2} : \BPP^2 \to \BPP^2$ ($[X_0, X_1, X_2] \to [X_0^q, X_1^q, X_2^q]$)
with $q \ge 2$. It lifts to an endomorphism $f_{\BFF_1} : \BFF_1 \to \BFF_1$
of $\deg(f_{\BFF_1}) = q^2$,
where $\BFF_1 \to \BPP^2$ is the blowup of the point $[0, 0, 1]$ fixed
by $f_{\BPP^2}^{-1}$.
}
\end{example}

\begin{remark}\label{rThA}
Below are some remarks about Theorem \ref{ThA}.
\begin{itemize}
\item[(1)]
The non-normal locus of $V_i$ ($i = 3, 4$) is a single line $C$
and stabilized by $f^{-1}$.
Let $\sigma: V_i' \to V_i$ ($i = 3, 4$) be the normalization.
Then $V_i'$ is the (smooth) Hirzebruch surface $\BFF_1$
(i.e., the one-point blowup of $\BPP^2$; see \cite[Theorem 1.5]{AF}, \cite{Re})
with the conductor $\sigma^{-1}(C) \subset V_i'$
a smooth section at infinity (for $V_3$), and the union of the negative section
and a fibre (for $V_4$), respectively. $f|V_i$ lifts to a (polarized)
endomorphism $f_{V_i'} : V_i' \to V_i'$.

\item[(2)]
$V_1$ (resp. $V_2$) is unique as a normal cubic
(or degree three del Pezzo) surface of Picard number one
and with the singular locus $\Sing V_1 = 3A_2$ (resp. $\Sing V_2 = E_6$);
see \cite[Theorem 1.2]{Ye}, and \cite[Theorem 4.4]{HW} for the anti-canonical
embedding of $V_i$ in $\BPP^3$.
$V_1$ contains exactly three lines of triangle-shaped whose three vertices
form the singular locus of $V_1$.
And $V_2$ contains a single line on which lies its unique singular point.
$f^{-3}$
fixes the singular point(s) of $V_i$ ($i = 1, 2$).

\item[(3)]
$f^{-1}$ (or its positive power) does not stabilize the only line $L$ on $V_2$
by using \cite[Theorem 4.3.1]{ENS} since the pair $(V_2, L)$ is not
log canonical at the singular point of $V_2$.
For $V_1$, we do not know whether $f^{-1}$ (or its power) stabilizes the three lines.
\end{itemize}
\end{remark}

\begin{setup}\label{set1.1} {\bf A motivating conjecture.}
Here are some motivations for our paper.
It is conjectured that {\it every hypersurface $V \subset \BPP^n$
stabilized by the inverse $f^{-1}$ of an endomorphism
$f: \BPP^n \to \BPP^n$ of $\deg(f) > 1$, is linear}.
This conjecture is still open when $n \ge 3$ and $V$ is singular, since
the proof of \cite{BCS} is incomplete as we were informed by an author.
The smooth hypersurface case was settled in the affirmative in any dimension
by Cerveau - Lins Neto \cite{CL} and independently by Beauville \cite{Be}.
See also \cite[Theorem 1.5 in arXiv version]{nz2}, \cite{uniruled} and \cite{P3endo}
for related results.

By Theorem \ref{ThA}, this conjecture is true when $n = 3$ but with
four exceptional cubic surfaces $V_i$ which we could not rule out.

From the dynamics point of view, as seen in Dinh-Sibony \cite[Theorem 1.3, Corollary 1.4]{DS},
$f : \BPP^n \to \BPP^n$ behaves nicely {\it exactly} outside
those $f^{-1}$-stabilized subvarieties.
We refer to Fornaess-Sibony \cite{FS}, and \cite{DS} for further references.
\end{setup}

A smooth hypersurface $X$ in $\BPP^{n+1}$ with $\deg(X) \ge 3$ and $n \ge 2$,
has no endomorphism $f_X : X \to X$ of degree $> 1$
(cf. \cite{CL}, \cite[Theorem]{Be}).
However, singular $X$ may have plenty of endomorphisms $f_X$ of arbitrary degrees
as shown in Example \ref{ex1ThA} below.
Conjecture \ref{set1.1} asserts that such $f_{X}$
can not be extended to an endomorphism of $\BPP^{n+1}$.

\begin{example}\label{ex1ThA}
{\rm
We now construct many polarized endomorphisms for some degree $n+1$ hypersurface
$X \subset \BPP^{n+1}$, with $X$ isomorphic to the $V_1$ in Theorem \ref{ThA}
when $n = 2$.
Let
$$f = (F_0, \dots, F_n) : \BPP^n \to \BPP^n$$ ($n \ge 2$),
with $F_i = F_i(X_0, \dots, X_n)$ homogeneous,
be any endomorphism of degree $q^n > 1$,
such that $f^{-1}(S) = S$ for a reduced degree $n+1$ hypersurface
$S = \{S(X_0, \dots, X_n) = 0\}$.
So $S$ must be normal crossing and linear: $S = \sum_{i=0}^{n} S_i$
(cf.~\cite[Theorem 1.5 in arXiv version]{nz2}).
Thus we may assume that $f = (X_0^q, \dots, X_n^q)$ and
$S_i = \{X_i = 0\}$.
The relation $S \sim (n+1)H$ with $H \subset \BPP^n$ a hyperplane,
defines
$$\pi : X = Spec \oplus_{i=0}^n \SO(-iH) \to \BPP^n$$
which is
a Galois $\BZZ/(n+1)$-cover
branched over $S$ so that $\pi^*S_i = (n+1)T_i$ with the restriction
$\pi|T_i : T_i \to S_i$ an isomorphism.

This $X$ is identifiable with the degree $n+1$ hypersurface
$$\{Z^{n+1} = S(X_0, \dots, X_n)\} \subset \BPP^{n+1}$$
and has singularity of type $z^{n+1} = xy$
over the intersection points of $S$ locally defined as $xy = 0$.
Thus, when $n = 2$, we have $\Sing X = 3A_2$ and
$X$ is isomorphic to the $V_1$ in Theorem \ref{ThA} (cf.~Remark \ref{rThA}).
We may assume that
$$f^*S(X_0,\dots, X_n) = S(X_0,\dots, X_n)^q$$ after replacing
$S(X_0, \dots, X_n)$
by a scalar multiple, so
$f$ lifts to an endomorphism
$$g = (Z^q, F_0, \dots, F_n)$$ of $\BPP^{n+1}$
(with homogeneous coordinates $[Z, X_0, \dots, X_n]$),
stabilizing $X$, so that
$$g_X := g|X : X \to X$$ is a polarized endomorphism
of $\deg(g_X) = q^n$ (cf.~
\cite[Lemma 2.1]{nz2}). Note that
$g^{-1}(X)$ is the union of $q$ distinct hypersurfaces
$$\{Z = \zeta^i S(X_0, \dots, X_n)\} \subset \BPP^{n+1}$$
(all isomorphic to $X$), where $\zeta := \exp( 2 \pi \sqrt{-1}/q)$.

This $X$ has only Kawamata log terminal singularities
and $\Pic X = (\Pic \BPP^{n+1}) \, | \, X$ ($n \ge 2$) is of rank one, using Lefschetz type theorem
\cite[Example 3.1.25]{La} when $n \ge 3$.
We have $f^{-1}(S_i) = S_i$ and $g_X^{-1}(T_i) = T_i$,
where $0 \le i \le n$.

When $n = 2$, the relation $(n+1)(T_1 - T_0) \sim 0$ gives rise to an \'etale-in-codimenion-one
$\BZZ/(n+1)$-cover
$$\tau: \BPP^{n} \isom \widetilde{X} \to X$$ so that $\sum_{i=0}^{n} \tau^*T_i$
is a union of $n+1$ normal crossing hyperplanes;
indeed, $\tau$ restricted over $X \setminus \Sing X$, is its universal cover
(cf.~\cite[Lemma 6]{MZ}),
so that $g_X$ lifts up to $\widetilde{X}$.
A similar result {\it seems} to be true for $n \ge 3$,
by considering the `composite' of the $\BZZ/(n+1)$-covers given by $(n+1)(T_i - T_0) \sim 0$
($1 \le i < n$).
}
\end{example}

\section{Proofs of Theorem \ref{ThA} and Remark \ref{rThA}}

We use the standard notation in Hartshorne's book and \cite{KM}.

\begin{setup}
We now prove Theorem \ref{ThA} and Remark \ref{rThA}.
By \cite[Theorem 1.5 in arXiv version]{nz2},
we may assume that $V \subset \BPP^3$ is an irreducible rational {\it singular} cubic hypersurface.

We first consider the case where $V$ is non-normal.
Such $V$
is classified in
\cite[Theorem 9.2.1]{Do} to the effect that either $V = V_i$ ($i = 3, 4$)
or $V$ is a cone over a nodal or cuspidal rational planar cubic curve $B$.
The description in Remark \ref{rThA} on $V_3, V_4$ and their normalizations, is given in
\cite[Theorem 1.1]{Re}, \cite[Theorem 1.5, Case (C), (E1)]{AF};
the $f^{-1}$-invariance of the non-normal locus $C$ is proved in \cite[Proposition 5.4 in arXiv version]{nz2}.

We consider and will rule out the case where $V$ is a cone over $B$.
Since $V$ is normal crossing in codimension $1$ (cf. \cite[Theorem 1.5 or Proposition 5.4 in arXiv version]{nz2}),
the base $B$ of the cone $V$ is nodal. Let $P$ be the vertex of the cone $V$,
and $L \subset V$ the generating line lying over the node of $B$.
Then $f_V := f|V$ satisfies
the assertion that $f_V^{-1}(P) = P$. Indeed, the normalization
$V'$ of $V$ is a cone over a smooth rational (twisted) cubic curve (in $\BPP^3$), i.e.,
the contraction of the $(-3)$-curve on the Hirzebruch surface $\BFF_3$ of degree $3$;
$f_V$ lifts to an endomorphism $f_{V'}$ of $V'$ so that
the conductor $C' \subset V'$ is preserved by $f_{V'}^{-1}$
(cf.~\cite[Proposition 5.4 in arXiv version]{nz2}) and
consists of two distinct generating lines
$L_i$ (lying over $L$). Thus
$f_{V'}^{-1}$ fixes the vertex $L_1 \cap L_2$ (lying over $P$). Hence $f_V^{-1}(P) = P$
as asserted.

By \cite[Lemma 5.9 in arXiv version]{nz2}, $f: \BPP^3 \to \BPP^3$ (with $\deg(f) = q^3 > 1$ say)
descends, via the projection
$\BPP^3 \ratmap \BPP^2$ from the point $P$, to an endomorphism $h : \BPP^2 \to \BPP^2$
with $\deg(h) = q^2 > 1$ so that $h^{-1}(B) = B$. This and $\deg(B) = 3 > 1$
contradict the linearity property of $h^{-1}$-stable curves in $\BPP^2$
(see e.g. Theorem 1.5 and the references in \cite[arXiv version]{nz2}).
\end{setup}

\begin{setup}\label{set3.1}
Next we consider the case where $V \subset \BPP^3$ is a normal rational {\it singular} cubic hypersurface.
By the adjunction formula,
$$-K_V = -(K_{\BPP^3} + V)|V \sim H|V$$
which is ample, where $H \subset \BPP^3$ is a hyperplane.
Since $K_V$ is a Cartier divisor, $V$ has only Du Val (or rational double, or
$ADE$) singularities.
Let
$$\sigma : V' \to V$$ be the minimal resolution.
Then
$$K_{V'} = \sigma^*K_V \sim \sigma^*(-H|V) .$$
For $f : \BPP^3 \to \BPP^3$,
we can apply the result below to $f_V := f|V$.
\end{setup}

\begin{lemma}\label{normal}
Let $V \subset \BPP^3$ be a normal cubic surface, and
$f_V : V \to V$ an endomorphism such that $f_V^*(H|V) \sim qH|V$
for some $q > 1$ and the hyperplane $H \subset \BPP^3$.
Let
$$S(V) = \{G \subset V \, | \, G: \text{irreducible}, G^2 < 0\}$$
be the set of negative curves on $V$,
and set
$$E_V := \sum_{E \in S(V)} \, E .$$
Replacing $f_V$
by its positive power, we have:

\begin{itemize}
\item[(1)]
If $f_V^*G \equiv a G$ for some Weil divisor $G \not\equiv 0$, then $a = q$.
We have
$$f_V^* (L|V) \sim q(L|V)$$ for every divisor $L$ on $\BPP^3$.
Especially, $\deg(f_V) = q^2$; $K_V \sim -H|V$  satisfies $f_V^*K_V \sim qK_V$.

\item[(2)]
$S(V)$ is a finite set. $f_V^*E = qE$ for every $E \in S(V)$. So $f_V^*E_V =qE_V$.

\item[(3)]
A curve $E \subset V$ is a line in $\BPP^3$ if and only if $E$ is equal to $\sigma(E')$
for some $(-1)$-curve $E' \subset V'$.

\item[(4)]
Every curve $E \in S(V)$ is a line in $\BPP^3$.

\item[(5)]
We have
$$K_V + E_V = f_V^*(K_V + E_V) + \Delta$$
for some effective divisor $\Delta$ containing no line in $S(V)$,
so that the ramification divisor
$$R_{f_V} = (q-1)E_V + \Delta .$$
In particular, the cardinality $\#S(V) \le 3$, and the equality
holds exactly when $K_V + E_V \sim_{\BQQ} 0$;
in this case, $f_V$ is \'etale outside the three lines of $S(V)$
and $f_V^{-1}(\Sing V)$.
\end{itemize}
\end{lemma}

\begin{proof}
For (1) and (2), we refer to \cite[Lemma 2.1]{nz2} and \cite[Proposition 3.6.8]{ENS}
and note that
$L \sim bH$ for some integer $b$.

(3) We may assume that $E' \isom \BPP^1$, where
$E' := \sigma'(E)$ is the proper transform of $E$.
(3) is true because $E$ is a line if and only if
$$1 = E . H|V \, ( = E' . \sigma^*(H|V)
= E' . (-K_{V'})),$$
and by the genus formula
$-2 = 2g(E') - 2 = (E')^2 + E' . K_{V'}$.

(4) $E' := \sigma'(E)$ satisfies
$E' . K_{V'} = E . K_V < 0$ and $(E')^2 \le E' . \sigma^*E = E^2 < 0$.
Hence $E'$ is a $(-1)$-curve by the genus formula.
Thus (4) follows from (3).

(5) The first part is true because, by (2), the ramification divisor
$R_{f_V} = (q-1)E_V +$ (other effective divisors).
Also, by (1) and (2), $\Delta \sim (1-q)(K_V + E_V)$.
Since $K_V . E = -1$ for every $E \in S(V)$ (by (4)),
we have
$$0 \le -K_V . \Delta = -K_V . (1-q)(K_V + E_V) = (q-1)(3 - \#S(V)) .$$
Now the second part of (5) follows from this and the fact that $\Delta = 0$
if and only if $-K_V . \Delta = 0$ since $-K_V$ is ample.
The last part of (5) follows from the purity of branch loci
and the description of $R_{f_V}$ in (5).
\end{proof}

\begin{setup}
We now prove Theorem \ref{ThA} and Remark \ref{rThA}
for the normal cubic surface $V$.
We use the notation in Lemma \ref{normal}.
Suppose that the Picard number
$$\rho := \rho(V) \ge 3 .$$
Since $K_V$ is not nef and by the minimal model program
for klt surfaces, there is a composite
$$V = V_{\rho} \overset{\tau_{\rho}}\to V_{\rho - 1} \cdots \overset{\tau_3}\to V_2$$
of birational extremal contractions such that
$$\rho(V_i) = i .$$
Let
$$E_i \subset V_i$$ be the exceptional (irreducible) divisor of $\tau_i: V_i \to V_{i-1}$.
Since $V$ is Du Val, either $E_i$ is contained in the smooth locus $V_i \setminus \Sing(V_i)$
and is a $(-1)$-curve, or $E_i$ contains exactly one singular point
$$P_i \in \Sing V_i$$
of type $A_{n_i}$ so that $\tau_i(E_i) \in V_{i-1}$ is a smooth point.
In particular, every $V_i$ is still Du Val.
Let
$$V_i' \to V_i$$ be the minimal resolution.
Since $-K_{V_i}$ is the pushforward of the ample divisor $-K_V$,
it is ample. So $V_i$ is still a Gorenstein del Pezzo surface.
Noting that $K_{V_i'}$ is the pullback of $K_{V_i}$, we have
$$(K_{V_{i-1}'}^2 =) \, K_{V_{i-1}}^2 = K_{V_i}^2 + (n_i+1) \ge 3 + (0 + 1) = 4$$
for all $3 \le i \le \rho$.

Note that the proper transform
$$E_i(V) \subset V$$ of $E_i \subset V_i$ is a negative curve.
Since $f_V^{-1}$ stabilizes every negative curve in $S(V)$
and especially $E_i(V)$
(when $f$ is replaced by its positive power, as seen in Lemma \ref{normal}),
$f_V$ descends to
$$f_i : V_i \to V_i .$$
The $V_2'$ and $S(V_2')$, the set of negative curves on $V_2'$, are classified in \cite[Figure 6]{Ye}.
Since $K_{V_2}^2 \ge 4$,
$(V_2', \, S(V_2'))$ is as described in one of the last $10$ cases in [ibid.].
For example, we write
$$V_2 = V_2(2A_2+A_1)$$ if $\Sing V_2$
consist of two points of type $A_2$ and one point of type $A_1$.

Except the four cases
$$V_2(D_4), \,\, V_2(4A_1), \,\, V_2(A_3), \,\, V_2(2A_1)$$ in [ibid.],
exactly two $(-1)$-curves in $S(V_2')$
map to intersecting negative curves
$$M_i \in S(V_2)$$
so that
$$S(V) = \{E_{\rho}, \, M_1(V), \, M_2(V)\}$$
with
$$M_i(V) \subset V$$ the proper transform of $M_i$, so
$$K_V + E_{\rho} + M_1(V) + M_2(V) \sim_{\BQQ} 0$$ (cf.  Lemma \ref{normal})
and hence $K_{V_2} + M_1 + M_2 \sim_{\BQQ} 0$,
which is impossible by a simple calculation and blowing down
$V_2'$ to its relative minimal model.

For each of the above four exceptional cases, we may assume that
$f_2^{-1}$ stabilizes both extremal rays $\BRR_+[M_i]$ of the closed cone $\overline{\NE}(V_2)$
of effective $1$-cycles, with
$$M_i \subset V_2$$ the image of some $(-1)$-curve on $V_2'$,
where both extremal rays are of fibre type in the cases $V_2(D_4)$ and $V_2(4A_1)$,
where the first (resp. second) is of fibre type (resp. birational type)
in the cases $V_2(A_3)$ and $V_2(2A_1)$.
Let
$$F_i \, (\sim 2M_i)$$ with $i = 1, 2$, or with $i = 1$ only, be the fibre of the
extremal fibration
$$\varphi_i = \Phi_{|2M_i|} : V \to B_i \isom \BPP^1$$ passing through the point
$(\tau_3 \circ \cdots \circ \tau_{\rho})(E_{\rho})$.
Then the proper transform
$$F_i(V) \subset V$$ of $F_i$ is a negative curve
so that
$$E_V = F_1(V) + F_2(V) + E_{\rho}$$ in the cases $V_2(D_4)$ and $V_2(4A_1)$, and
$$E_V = F_1(V) + M_2(V) + E_{\rho}$$ in the cases $V_2(A_3)$ and $V_2(2A_1)$
(cf. Lemma \ref{normal}). Then $K_V + E_V \sim_{\BQQ} 0$,
and hence $K_{V_2} + F_1 + F_2 \sim_{\BQQ} 0$
or $K_{V_2} + F_1 + M_2 \sim_{\BQQ} 0$ where the latter is impossible by a simple calculation
as in the early paragraph. Thus
$$K_{V_2} + F_1+ F_2 \sim_{\BQQ} 0 .$$
By making use of Lemma \ref{normal} (1) or (2), $f_2^*F_i = qF_i$, and $f_2$ descends to
an endomorphism
$$f_{B_1} : B_1 \to B_1$$ of degree $q$.
Thus the ramification divisor of $f_{B_1}$ is of degree $2(q-1)$ by the Hurwitz formula,
and is hence equal to
$$(q-1)P + \sum (b_i-1)P_i$$ with
$$\sum (b_i-1) = q-1$$
where $P \in B_1$ so that $F_1$ lies over $P$. But then
$$R_{f_2} \ge (q-1)(F_1+F_2) + \sum (b_i-1)F_i'$$ where $F_i'$ are fibres of $\varphi_1$
lying over $P_i$, so that
$$K_{V_2} + F_1 + F_2 \ge f_{2}^*(K_{V_2} + F_1 + F_2) + \sum (b_i-1) F_i'$$
which is impossible since $K_{V_2} + F_1 + F_2 \sim_{\BQQ} 0$.
\end{setup}

\begin{setup}\label{rho2}
Consider the case $\rho(V) = 2$. Then the minimal resolution 
$$V' \to V$$
and its negative curves are described in one of the first five cases in
\cite[Figure 6]{Ye}.

For the case $V = V(A_5)$, two $(-1)$-curves on $V'$ map to two negative curves
$$M_1, \,\, M_2$$ on $V$.
Note that $f_V^*(M_i) = q^*M_i$ (see Lemma \ref{normal}).
There is a contraction
$$V \to \BPP^2$$ of $M_1$ so that the image of $M_2$ is a
plane conic preserved by $f_P^{-1}$ where
$$f_P : \BPP^2 \to \BPP^2$$ is the descent
of $f_V$ (of degree $q^2 > 1$), contradicting \cite[Theorem 1.5(4) in arXiv version]{nz2}.

For the case $V(2A_2+A_1)$,
there are exactly five $(-1)$-curves
$$M_i' \, \subset \, V'$$
with
$M_i \subset V$ their images.
Moreover, $M_1' . M_2' = 1$ and both $M_i$ ($i = 1, 2$) are negative curves on $V$;
each $M_j'$ ($j = 3, 4$) meets the isolated $(-2)$-curve;
$M_1$ and $M_3$ (reap. $M_2$ and $M_4$) meet the same component
of one (resp. another) $(-2)$-chain of type $A_2$.
We have
$$M_1 + M_2 \sim 2L$$ for some integral Weil divisor $L$,
by considering a relative minimal model of $V'$.
In fact, $M_1 + 3M_2 \sim 4M_3$.
Since $f_V^*(M_1+M_2) = q(M_1+M_2)$ (see Lemma \ref{normal}),
$f_V$ lifts to some
$$g : U \to U .$$ Here the double cover (given by the relation $M_1+M_2 \sim 2L$)
$$\pi : U = Spec \oplus_{i=0}^1 \SO(-iL) \to V$$
is branched along $M_1 + M_2$.
Indeed, when $2 \, \not| \, q$, the normalization
$$\hat{U}$$ of the fibre product of $\pi : U \to V$
and $f_V : V \to V$ is isomorphic to $U$ and we take $g$ to be the first projection
$\hat{U} \to U$;
when $2 \, | \, q$, we have $\hat{U} = V \coprod V$ and let $g$ be the composite
of $\pi: U \to V$, the inclusion $V \cup \emptyset \to V \coprod V$
and the first projection $\hat{U} \to U$. Now $\Sing U$ consists of
a type $A_1$ singularity lying over $M_1 \cap M_2$ and four points in $\pi^{-1}(\Sing V)$
of type $A_1, A_1$, $\frac{1}{3}(1,1)$ and $\frac{1}{3}(1,1)$,
and every $M_j$ ($j = 3, 4$) splits into two negative
curves on $U$ which are hence preserved by $g^{-1}$ (as in Lemma \ref{normal}
after $f_V$ is replaced by its positive power).
Thus $f_V^{-1}(M_i) = M_i$ ($1 \le i \le 4$).
As in the proof of Lemma \ref{normal}, $E_V' := M_1 + M_2 + M_3$
satisfies $K_V + E_V' \sim_{\BQQ} 0$ and
$f_V$ is \'etale over $V \setminus (E_V' \cup \Sing V)$.
The latter contradicts the fact that $f_V^*M_4 = qM_4$ and hence $f_V$
has ramification index $q$ along $M_4$.

For $V = V(A_4+A_1)$, there are exactly four $(-1)$-curves
$$M_i' \, \subset \, V'$$
with
$M_i \, \subset \, V$ their images. We may so label that the five $(-2)$-curves and $M_j'$ ($j = 1, 2, 3$)
form a simple loop:
$$M_1' - (-2) - (-2) - (-2) - (-2) - M_2' - M_3' - (-2) - M_1' ;$$
both  $M_j$ ($j = 2, 3$) are negative curves on $V$.
We can verify that $M_3 + 2M_2 \sim 3(M_5 - M_2)$,
and $M_1 + M_2 \sim M_5$, where $M_5$ is the image of a curve $M_5' \isom \BPP^1$
and $M_5'$ is the smooth fibre (passing through the intersection point of $M_3'$
and the isolated $(-2)$-curve) of the $\BPP^1$-fibration on $V'$
with a singular fibre consisting of $M_1'$, $M_2'$ and the $(-2)$-chain of type $A_4$
sitting in between them.
As in the case $V(2A_2+A_1)$, $f_V$ lifts to
$$g : U \to U$$ on the triple cover
$U$ defined by the relation $M_3 + 2M_2 \sim 3(M_5 - M_2)$,
so that each $M_i$ ($i = 4, 5$)
splits into three negative curves on $U$ preserved by $g^{-1}$.
Hence $f_V^{-1}(M_i) = M_i$ ($i = 2, \dots, 5$). But then for $E_V' := M_2 + M_3 + M_4$
we have $K_V + E_V' \sim_{\BQQ} 0$ as in the proof of Lemma \ref{normal}
so that $f_V$ is \'etale over $V \setminus (E_V' \cup \Sing V)$,
contradicting the fact that $f_V$ has ramification index $q$ along $M_5$.

For $V = V(D_5)$, the lonely $(-1)$-curve $M_1'$ and the intersecting $(-1)$-curves
$M_2' \cup M_3'$ on $V'$ satisfy $2M_1 \sim M_2 + M_3$ where
$$M_i \, \subset \, V$$
denotes the image of $M_i'$ (indeed, the three $M_i'$ together with
the five $(-2)$-curves form the support of two singular fibres and a section in a
$\BPP^1$-fibration).
As in the case $V(2A_2+A_1)$, $f_V$ lifts to
$$g : U \to U$$ on the double cover
$U$ defined by the relation $2M_1 \sim M_2 + M_3$,
so that $M_1$ splits into two negative curves on $U$ preserved by $g^{-1}$.
Thus $f_V^{-1}(M_1) = M_1$.
Hence $(V, M_1)$ is log canonical (cf. \cite[Theorem 4.3.1]{ENS}).
But $(V, M_1)$ is not log canonical because
$M_1'$ meets the $(-2)$-tree of type-$D_5$ in a manner different from the
classification of \cite[Theorem 9.6]{Ka}.
We reach a contradiction.

For $V = V(A_3 + 2A_1)$, let $M_1' \subset V'$ be the $(-1)$-curve
meeting the middle component of the $(-2)$-chain of type $A_3$,
let $M_3'$ be the $(-1)$-curve meeting two isolated $(-2)$-curves,
and let $M_2'$ be the $(-1)$-curve meeting both $M_1'$ and $M_3'$.
Then the images
$$M_i \, \subset \, V$$ of $M_i'$ satisfy
$2M_1 \sim 2M_3$
(indeed, $M_1'$, $M_3'$ and the five $(-2)$-curves form the support of two singular fibres in some
$\BPP^1$-fibration).
The relation $2(M_1 - M_3) \sim 0$ defines a double cover
$$\pi : U = Spec \oplus_{i=0}^1 \SO(-i(M_1-M_3)) \to V$$
\'etale over $V \setminus \Sing V$.
In fact, $\pi$ restricted over $V \setminus \Sing V$,
is the universal cover over it,
so $U$ is again a Gorenstein del Pezzo surface and hence the
irregularity $q(U) = 0$.
Thus $f_V$ lifts to
$$g : U \to U .$$
Now $\pi^{-1}(\Sing V)$ consists of two smooth points and the unique singular point
of $U$ (of type $A_1$), and
each $\pi^*M_i$ ($i = 1, 2$) splits into two negative curves $M_{i}(1)$, $M_{i}(2)$ on $U$
preserved by $g^{-1}$;
thus $f_V^*M_i = qM_i$ and $g^*M_i(j) = qM_i(j)$ (as in Lemma \ref{normal} (1)).

We assert that $f_V^{-1}$ permutes members of the pencil
$$\Lambda := |M_1 + M_2| .$$
It suffices to show that $g^{-1}$ permutes members of the irreducible pencil
$$\Lambda_U$$
(parametrized by $\BPP^1$ for $q(U) = 0$)
which is the pullback of $\Lambda$. Now $\pi^*(M_1+M_2)$ splits into two
members $$M_1(j) + M_2(j) = \divv(\xi_j)$$ (in local equation; $j = 1, 2$) which are preserved by $g^{-1}$
and span $\Lambda_U$.
We may assume that $g^* \xi_j = \xi_j^q$ after replacing the equation by a scalar multiple.
Then the $g^*$-pullback of every member $\divv(a \xi_1 + b \xi_2)$ in $\Lambda_U$
is equal to $\divv(a \xi_1^q + b \xi_2^q)$ and hence is the union of members in $\Lambda_U$
because we can factorize $a \xi_1^q + b \xi_2^q$ as a product of linear forms
in $\xi_1$, $\xi_2$. This proves the assertion.

By the assertion and since $f_V^*(M_1+M_2) = q(M_1+M_2)$,
$f_V$ descends to an endomorphism
$$f_B : B \to B$$ of degree $q$
on the curve $B \isom \BPP^1$ parametrizing the pencil $\Lambda$.
We have $f_B^*P_0 = qP_0$ for the point parametrizing the member $M_1+M_2$ of $\Lambda$.
Write $K_B = f_B^*K_B + R_{f_B}$, where the ramification divisor
$$R_{f_B} = (q-1) P_0 + \Delta_B$$ with
$\Delta_B = \sum (b_i-1) Q_j$ of degree $q-1$ for some $b_i \ge 2$.
Thus the ramification divisor
$$R_{f_V} = (q-1)(M_1 + M_2) + \Delta_V$$
with $\Delta_V = \sum (b_i-1) F_i +$ (other effective divisor),
where
$$F_i \in \Lambda$$ is parametrized by $Q_i$.
On the other hand, one can verify that $-K_V \sim 2M_1 + M_2$,
by blowing down to a relative minimal model of $V'$;
indeed, $M_2$ is a double section of the $\BPP^1$-fibration
$$\varphi := \Phi_{|2M_1|} : V \to \BPP^1 .$$
So $-M_1 \sim K_V + M_1 + M_2 = f_V^*(K_V + M_1 + M_2) + \Delta_V$
and hence, by Lemma \ref{normal} (1),
$$(b_1-1) F_1 \le \Delta_V \sim (1-q) (K_V + M_1 + M_2) \sim (q-1)M_1 .$$
This is impossible because $F_1$ is horizontal to the half fibre $M_1$ of $\varphi$.
Indeed, $F_1 . M_1 = (M_1+M_2) . M_1 = M_2 . M_1 = 1$.
\end{setup}

\begin{setup}\label{rho1}
Consider the last case $\rho(V) = 1$. Since $K_V^2 = 3$,
we have
$$V = V(3A_2), \,\, V(E_6), \,\, \text{or} \,\, V(A_1+A_5) ,$$ and
the minimal resolution
$$V' \to V$$
and the negative curves on $V'$ are described in \cite[Figure 5]{Ye}.
For the first two cases, $V$ is isomorphic to $V_i$ ($i = 1$, or $2$) in Theorem \ref{ThA}
by the uniqueness result in \cite[Theorem 1.2]{Ye} and by \cite[Theorem 4.4]{HW}.

For $V = V(A_1+A_5)$, the images
$$M_i \subset V$$ of the two $(-1)$-curves
$M_i' \subset V'$ satisfy $2M_1 \sim 2M_2$;
indeed, $M_i'$ together with
the six $(-2)$-curves form the support of two singular fibres and a section in some
$\BPP^1$-fibration.
Let
$$\pi : U \to V$$ be the double cover given by the relation $2(M_1 - M_2) \sim 0$.
In fact, $\pi$ restricted over $V \setminus \Sing V$,
is the universal cover over it. So $f_V$ lifts to
$$g : U \to U .$$
As in the case $V(A_3+2A_1)$, if we let $M_1'$ be the one meeting the second
component of the $(-2)$-chain of type $A_5$, then
$\pi^*M_1$ splits into two negative curves on $U$ preserved by $g^{-1}$.
Thus $f_V^{-1}(M_1) = M_1$, and, as in the case $V(D_5)$ above,
contradicts \cite[Theorem 4.3.1]{ENS} and \cite[Theorem 9.6]{Ka}.

This completes the proof of Theorem \ref{ThA} for normal cubic surfaces
and hence the whole of Theorem \ref{ThA}.
To determine the equations of $V_i$ ($i = 1, 2$), we can check that the
equations in Theorem \ref{ThA} possess the
right combination of singularities and then use the very ampleness of $-K_{V_i}$
to embed $V_i$ in $\BPP^3$ as in \cite{HW}
and the uniqueness of $V_i$ up to isomorphism, and hence
up to projective transformation by \cite{HW} (cf.~\cite[Theorem 1.2]{Ye}).
\end{setup}

\begin{setup}\label{*}
Now we prove Remark \ref{rThA}.
From Lemma \ref{normal} till now, we did not assume the hypothesis $(*)$ that
$f_V$ is the restriction of some $f : \BPP^3 \to \BPP^3$
whose inverse stabilizes $V$.
From now on till the end of the paper, we assume
this hypothesis $(*)$.

For $V = V(E_6)$ or $V(3A_2)$, the relation
$V \sim 3H$ defines a triple cover
$$\pi : X = Spec \oplus_{i=0}^2 \SO(-iH) \to V$$
branched along $V$.
Then
$$X = \{Z^3 = V(X_0, \dots, X_3)\} \subset \BPP^4$$
is a cubic hypersurface, where we let $V(X_0, \dots, V_3)$ be the cubic
form defining $V \subset \BPP^3$.
Our $\pi^{-1}$ restricts to a bijection
$\pi^{-1} : \Sing V \to \Sing X$.
As in Example \ref{ex1ThA}, $f$ lifts to $(Z^q, f) : \BPP^4 \to \BPP^4$ stabilizing $X$,
so that the restriction
$$g = (Z^q, f)|X : X \to X$$ is also a lifting of $f$.
By the Lefschetz type theorem \cite[Example 3.1.25]{La},
$\Pic(X) = (\Pic(\BPP^4))|X$.

For $V = V(E_6)$, $V'$ contains only one $(-1)$-curve $M'$
and hence $V$ contains only one line $M$ (the image of $M'$) by Lemma \ref{normal}.
Note that
$$\{Q\} := \Sing V \subset M .$$
Let
$$\Pi_{MM} \subset \BPP^3$$ (say given by $X_3 = 0$)
be the unique plane such that
$$\Pi_{MM} | V = 3M .$$
Indeed, $3M$ belongs to the complete linear system $|H|V|$, and the exact sequence
$$0 \rightarrow \SO(-2H) \rightarrow \SO(H) \rightarrow \SO_V(H) \rightarrow 0$$
and the vanishing of $H^1(\BPP^3, -2H)$ (e.g. by the Kodaira vanishing)
imply
$$H^0(\BPP^3, \SO(H)) \isom H^0(V, \SO_V(H)) .$$
Our $\pi^* \Pi_{MM}$ is a union of three $2$-planes
$$L_i \, \subset \, \BPP^4$$
because the restriction of $\pi$ over $\Pi_{MM}$ is given
by the equation
$$Z^3 = V(X_0, \dots, X_3) \, | \, \Pi_{MM} = M(X_0, \dots, X_2)^3$$
where $M(X_0, \dots, X_2)$ is a linear equation of $M \subset \Pi_{MM}$.
This and the fact that $\pi^* \Pi_{MM}$ is a generator of $\Pic(X) = (\Pic(\BPP^3)) \, | \, X$,
imply that the Weil divisor $L_1$ is not a Cartier divisor on $X$.
Since $\Sing X$ consists of a single point $P$ lying over $\{Q\} = \Sing V$,
$L_1$ is not Cartier at $P$ and hence $X$ is not factorial at $P$.
Thus $X$ is not $\BQQ$-factorial at $P$ because the local $\pi_1$ of $P$
is trivial by a result of Milnor (cf. the proof of \cite[Lemma 5.1]{Ka}).
Hence $g^{-1}(P)$ contains no smooth point (cf. \cite[Lemma 5.16]{KM})
and must be equal to $\Sing X = \{P\}$. Thus $f^{-1}(Q) = Q$
because $\pi^{-1}(Q) = P$.
\end{setup}

\begin{setup}\label{3A2}
Before we treat the case $V(3A_2)$, we make some remarks.
Up to isomorphism, there is only one $V(3A_2)$ (cf. \cite[Theorem 1.2]{Ye}).
Set $V := V(3A_2)$.
There is a Gorenstein del Pezzo surface $W$ such that $\rho(W) = 1$,
$\Sing W$ consists of four points $\beta_i$ of Du Val type $A_2$,
$$\pi_1(W \setminus \Sing W) = (\BZZ/(3))^{\oplus 2}$$ and
there is a Galois triple cover $V \to W$ \'etale over $W \setminus \{\beta_1, \beta_2, \beta_3\}$
so that a generator
$$h \in \Gal(V/W)$$ permutes the three singular points of $V$
lying over $\beta_4$ (cf. \cite[Figure 1, Lemma 6]{MZ}).
Since the embedding $V \subset \BPP^3$ is given by the complete linear system
$|-K_V|$ (cf.~\cite{HW})
and $h^*(-K_V) \sim -K_V$, our $h$ extends to a projective transformation
of $\BPP^3$, also denoted as $h$.
Since $h(V) = V$,
$$h^*V(X_0, \dots, X_3) = cV(X_0, \dots, X_3)$$
for some nonzero constant $c$. This $h$ lifts to a projective transformation of $\BPP^4$,
also denoted as $h$, stabilizing the above triple cover $X \subset \BPP^4$ of $\BPP^3$
by defining $h^*Z = \root{3}\of{c} Z$. Then this $h$ permutes the three singular points
of $X$ lying over $\Sing V$.
\end{setup}

\begin{setup}
For $V = V(3A_2)$, $V'$ has exactly three $(-1)$-curves $M_i'$
and their images $M_i$ are therefore the only lines on $V$ (cf. Lemma \ref{normal}).
The graph $\sum M_i$ is triangle-shaped whose vertices
(the intersection $M_i \cap M_j$) are the three points in $\Sing V$. The sum of the
three $(-1)$-curves $M_i'$ and three $(-2)$-chains of type $A_2$
is linearly equivalent to $-K_{V'}$ and hence $\sum M_i \sim -K_V$;
also $2M_a \sim M_b+M_c$ so long $\{a, b, c\} = \{1, 2, 3\}$;
indeed, the three $M_i'$ and the six $(-2)$-curves form the support
of two singular fibres and two cross-sections of some $\BPP^1$-fibration.
Thus $3M_i \sim -K_V \sim H|V$. As argued in the case $V(E_6)$,
there is a unique hyperplane $\Pi_i$ such that
$$\Pi_i | V = 3M_i ;$$
our $\pi^* \Pi_{i}$ is a union of three $2$-planes $L_{ij}$ in $\BPP^4$
(sharing a line lying over $M_i \subset \BPP^3$),
$L_{i1}$ is not a Cartier divisor on $X$,
the $X$ is not $\BQQ$-factorial at least at one of the two points (and hence at
both points, since the above $h$ permutes $\Sing X$) in $L_{i1} \cap \Sing X$
(lying over $M_i \cap \Sing V$), and
$g^{-1}(\Sing X) = \Sing X$. Thus $f^{-1}(\Sing V) = \Sing V$.
Hence $f^{-3}$ fixes each point in $\Sing V$.

This completes the proof of Remark \ref{rThA} for normal cubic surfaces
and hence the whole of Remark \ref{rThA}.
\end{setup}

\begin{remark}
The proof of Theorem \ref{ThA} actually shows:
if $f_V : V \to V$ is an endomorphism
(not necessarily the restriction of some $f: \BPP^3 \to \BPP^3$)
of $\deg(f_V) > 1$ of a
Gorenstein normal del Pezzo surface with $K_V^2 = 3$
(i.e., a normal cubic surface), then $V$ is equal to $V_1$ or $V_2$ in
Theorem \ref{ThA} in suitable projective coordinates.
\end{remark}


\begin{thebibliography}{99}

\bibitem{AF}
M. Abe and M. Furushima, On non-normal del Pezzo
surfaces, Math.\ Nachr.\ \textbf{260} (2003), 3--13.

\bibitem{Be} A.~Beauville,
Endomorphisms of hypersurfaces and other manifolds,
Internat.\ Math.\ Res.\ Notices 2001, no.~1, 53--58.

\bibitem{BCS}
J.-V.~Briend, S.~Cantat and M.~Shishikura,
Linearity of the exceptional set for maps of \(\BPP_{k}(\BCC)\),
Math.\ Ann.\  \textbf{330}  (2004), 39--43.

\bibitem{CL}
D.~Cerveau and A.~Lins Neto, Hypersurfaces exceptionnelles
des endomorphismes de \(CP(n)\),
Bol.\ Soc.\ Brasil.\ Mat.\ (N.S.)  \textbf{31}  (2000),  no. 2, 155--161.

\bibitem{DS}
T. -C.~Dinh and N.~Sibony,
Equidistribution speed for endomorphisms of projective spaces, 
Math. \  Ann. \ \textbf{347} (2010), no. 3, 613--626.

\bibitem{Do}
I.~V.~Dolgachev,
Topics in Classical Algebraic Geometry.
Part I, January 19, 2009,
at: http://www.math.lsa.umich.edu/~idolga/topics1.pdf

\bibitem{FS} J.~E.~Fornaess and N.~Sibony,
Complex dynamics in higher dimension. I,
\emph{Complex analytic methods in dynamical systems}
(Rio de Janeiro, 1992), pp.~201--231,
Ast\'erisque \textbf{222}, Soc.\ Math.\ France, 1994.

\bibitem{HM}
C.~D.~Hacon and J.~ McKernan,
On Shokurov's rational connectedness conjecture, Duke Math. \ J. \textbf{138} (2007), no. 1, 119--136.

\bibitem{HW}
F.~Hidaka and K.~Watanabe,
Normal Gorenstein surfaces with ample anti-canonical divisor,
Tokyo \ J. \ Math. \ \textbf{4} (1981), no. 2, 319--330.

\bibitem{Ka} Y.~Kawamata,
Crepant blowing-ups of $3$-dimensional
canonical singularities and its application to degeneration of surfaces,
Ann. \ of Math. \ \textbf{127} (1988), 93--163.

\bibitem{KM} J.~Koll\'ar and S.~Mori,
Birational geometry of algebraic varieties,
Cambridge Tracts in Math. \textbf{134},
Cambridge Univ.\ Press, 1998.

\bibitem{La} R.~Lazarsfeld,
Positivity in algebraic geometry. I. Classical setting: line
bundles and linear series. Ergebnisse der Mathematik und ihrer
Grenzgebiete. 3. Folge.
A Series of Modern Surveys in Mathematics,
Vol. \ \textbf{48}. Springer-Verlag, Berlin, 2004.

\bibitem{MZ}
M.~Miyanishi and D. -Q.~Zhang,
Gorenstein log del Pezzo surfaces of rank one,
J. Algebra, \textbf{118} (1988), 63-84.

\bibitem{ENS}
N.~Nakayama, On complex normal projective surfaces
admitting non-isomorphic surjective endomorphisms,
Preprint 2 September 2008.

\bibitem{nz2}
N.~Nakayama and D. -Q.~Zhang,
Polarized endomorphisms of complex normal varieties,
 Math. \ Ann. \ \textbf{346} (2010), no. 4, 991--1018; also: arXiv:0908.1688v1.

\bibitem{Re} M.~Reid,
Nonnormal del Pezzo surfaces,
Publ.\ Res.\ Inst.\ Math.\ Sci.\ Kyoto Univ.\
\textbf{30} (1994), 695--727.

\bibitem{Ye} Q.~Ye,
On Gorenstein log del Pezzo surfaces, Japan. \ J. \ Math. \ \textbf{28} (2002), no. 1, 87--136.

\bibitem{uniruled}
D. -Q.~Zhang,
Polarized endomorphisms of uniruled varieties (with an Appendix by Y. Fujimoto and N. Nakayama),
Compos. \ Math. \ \textbf{146} (2010) 145 - 168.

\bibitem{P3endo} D. -Q.~Zhang,
Invariant hypersurfaces of endomorphisms of projective varieties, Preprint 2009.

\end{thebibliography}
\end{document}